\documentclass[11pt,reqno]{amsart}
\usepackage{latexsym,amscd}
\usepackage{amsmath,amssymb,amsthm,url,mathrsfs}
\usepackage{enumerate,stmaryrd,scalefnt}
\usepackage{color}

\usepackage[colorlinks=true,urlcolor=black,linkcolor=black,citecolor=black]{hyperref}

\theoremstyle{plain}

\newtheorem{lemma}{Lemma}

\theoremstyle{definition}

\newcounter{claim}

\newcounter{conjecture}

 \theoremstyle{remark}

\renewcommand{\phi}{\ensuremath{\varphi}}

\newcommand{\lb}{\ensuremath{\llbracket}}
\newcommand{\rb}{\ensuremath{\rrbracket}}

\newcommand{\rel}[1]{\ensuremath{\mathbin{#1}}}

\renewcommand{\leq}{\ensuremath{\leqslant}}

\newcommand{\meet}{\ensuremath{\wedge}}
\newcommand{\join}{\ensuremath{\vee}}

\newcommand{\Eq}{\ensuremath{\operatorname{Eq}}}

\begin{document}
\title{Dedekind's Transposition Principle for Lattices of Equivalence Relations}
\date{November 14, 2012}
\author{William DeMeo}
\email{williamdemeo@gmail.com}
\urladdr{\url{http://williamdemeo.wordpress.com}}
\address{Department of Mathematics\\
University of South Carolina\\Columbia 29208\\USA}

\begin{abstract}
We prove a version of Dedekind's Transposition
Principle that holds in  lattices of equivalence relations.
\end{abstract}

\maketitle
\setcounter{section}{+1}
\setcounter{secnumdepth}{2}

In this note we prove a version of Dedekind's Transposition
Principle\footnote{If $L$ is a modular lattice, then for any two elements $a, b\in L$
the intervals $\lb b, a\join b\rb$ and $\lb a\meet b, a\rb$ are isomorphic.
 See \cite{Dedekind:1900}, or~\cite[page 57]{alvi:1987}.} that
holds in all (not necessarily modular) lattices of equivalence 
relations.
Let $X$ be a set and let $\Eq X$ denote the lattice of equivalence relations on $X$.  Given $\alpha,
\beta \in \Eq X$, we define the
\emph{interval sublattice of equivalence relations above $\alpha$ and below
  $\beta$}, denoted $\lb \alpha, \beta \rb$, as follows:
\[
\lb \alpha, \beta\rb:= \{\gamma \in \Eq X \mid \alpha \leq \gamma \leq \beta\}.
\]
Let $L$ be a sublattice of $\Eq X$.  Given $\alpha,
\beta \in L$, let $\lb \alpha, \beta \rb_L:= \lb \alpha, \beta \rb \cap L$,
which we call an \emph{interval sublattice of $L$}, or more simply, an \emph{interval of $L$}.
Given 
$\alpha, \beta, \theta \in L$, let  $\lb \alpha, \beta \rb_L^\theta$ denote the
set of equivalence relations in the interval $\lb \alpha, \beta \rb_L$  that permute with
$\theta$.  That is, 
\[
\lb \alpha, \beta \rb_L^\theta:= \{\gamma \in L \mid \alpha \leq \gamma \leq
\beta \text{ and } \gamma \circ \theta = \theta \circ \gamma\}.
\]
\begin{lemma}
\label{lem:1}
If $\eta, \theta \in L\leq \Eq X$, and if $\eta\circ \theta = \theta \circ \eta$, then
\[\lb\theta, \eta \join \theta\rb_L \cong \lb\eta \meet \theta, \eta\rb_L^\theta \leq 
\lb\eta \meet \theta, \eta\rb_L.
\]
\end{lemma}
The lemma states that the sublattice $\lb\theta, \eta \join \theta\rb_L$ is isomorphic to the 
lattice, $\lb\eta \meet \theta, \eta\rb_L^\theta$, of relations in $L$ that are below
$\eta$, above $\eta \meet \theta$, and permute with $\theta$; moreover, $\lb\eta \meet \theta,
\eta\rb_L^\theta$ is a sublattice of
$\lb\eta \meet \theta, \eta\rb_L$.
To prove this, we need the following generalized version of \emph{Dedekind's 
Rule}:\footnote{In the group theory setting, the well known Dedekind's
  Rule states that if $A, B, C$ are subgroups of a group,
  and $A\leq B$, then we have the following identity of sets: $A(B\cap C) = B
  \cap AC$.}
\begin{lemma}
\label{lem:dedekind}
If $\alpha, \beta, \gamma \in L \leq \Eq X$, and if $\alpha
\leq \beta$, then we have the following identities of subsets of $X^2$:
\begin{equation}
  \label{eq:1}
  \alpha \circ (\beta \cap \gamma) = \beta \cap (\alpha \circ \gamma),
\end{equation}
\begin{equation}
  \label{eq:2}
  (\beta \cap \gamma) \circ \alpha = \beta \cap (\gamma \circ \alpha).
\end{equation}
\end{lemma}
\begin{proof}
We prove~(\ref{eq:1}); the proof of (\ref{eq:2}) is similar.  
First we check that
$\alpha \circ (\beta \cap \gamma) \subseteq \beta \cap (\alpha \circ \gamma)$. 
Indeed, since $\alpha\leq \beta$, we have
\[
\alpha \circ (\beta \cap \gamma) \subseteq \alpha \join (\beta \cap \gamma) \leq
\beta \join (\beta \cap \gamma) = \beta.
\]
Also, $\beta\cap \gamma \leq \gamma$ implies 
$\alpha \circ (\beta \cap \gamma) \subseteq
\alpha \circ \gamma$.  Therefore, 
 $\alpha \circ (\beta \cap \gamma) \subseteq \beta \cap (\alpha \circ \gamma)$. 

For the reverse inclusion, fix $(x,y) \in \beta \cap (\alpha\circ \gamma)$.
Since $(x,y) \in \alpha\circ \gamma$, there exists $c\in X$ such that $x\rel{\alpha} c
\rel{\gamma} y$.  We must produce $d\in X$ such that 
$x \rel{\alpha} d \rel{(\beta\cap \gamma)} y$.  
In fact, $d = c$ works, since 
$(x,c) \in \alpha \leq \beta$ implies $c \rel{\beta} x \rel{\beta} y$, so $(c,y) \in
\beta\cap \gamma$. 
\end{proof}

\begin{proof}[Proof of Lemma~\ref{lem:1}]
Let $\eta, \theta \in L\leq \Eq X$ be permuting equivalence relations in $L$, 
so $\eta \circ \theta = \theta \circ \eta = \eta \join \theta$.
Consider the mapping
$\phi : \lb\theta, \eta \join \theta\rb_L \rightarrow \lb\eta \meet \theta,\eta\rb$ 
given by 
$\alpha \mapsto \alpha \meet \eta$.  Clearly $\phi$ maps $\lb\theta, \eta \join \theta\rb_L$ into the sublattice
$\lb\eta \meet \theta,\eta\rb_L \leq \lb\eta \meet \theta,\eta\rb$.  Moreover,
it's easy to see that the range of $\phi$
consists of elements of $L$ that permute with $\theta$, so that 
$\phi$ maps 
into $\lb\eta \meet \theta, \eta\rb_L^\theta$.
Indeed, if $\alpha \in \lb \theta, \eta 
\join \theta\rb_L$, then by 
Lemma~\ref{lem:dedekind} we have
$(\alpha \meet \eta) \circ \theta = \alpha \cap (\eta\circ \theta) = \alpha \cap
(\theta \circ \eta) = \theta \circ (\alpha \meet \eta)$.

Next, consider the mapping $\psi : \lb \eta\meet \theta, \eta \rb_L^\theta
\rightarrow \lb\theta, \eta\join \theta\rb$ given by $\psi(\alpha) = \alpha
\circ \theta$.  Note that $\psi(\alpha) = \alpha \circ \theta = \alpha \join
\theta$, an element of $L$, since the domain of $\psi$ is a set of
relations in $L$ that permute with $\theta$.
We show that the two maps
\begin{equation}
  \label{eq:phi}
\phi : \lb\theta, \eta \join \theta\rb_L \ni \alpha \longmapsto 
\alpha \meet \eta  \in \lb\eta \meet \theta, \eta\rb_L^\theta
\end{equation}
\begin{equation}
  \label{eq:psi}
\psi :  \lb\eta \meet \theta, \eta\rb^\theta_L \ni \alpha \longmapsto \alpha \circ \theta \in
\lb\theta, \eta \join \theta\rb_L.
\end{equation}
are inverse lattice isomorphisms.  It is clear that these maps are order preserving.
Also, for $\alpha \in \lb\theta, \eta \join \theta\rb_L$ we have, by Lemma~\ref{lem:dedekind},
$\psi\, \phi(\alpha) = (\alpha \meet \eta)\circ \theta = 
\alpha \cap (\eta\circ \theta) = 
\alpha \cap (\eta\join \theta) = \alpha$. For
$\alpha \in \lb\eta \meet \theta, \eta\rb_L^\theta$, we have, by Lemma~\ref{lem:dedekind},
$\phi\, \psi(\alpha) = \phi(\alpha\circ \theta) =  (\alpha
\circ \theta) \meet \eta = \alpha \circ (\theta \meet \eta)$.

To complete the proof of Lemma~\ref{lem:1}, we show that
$\lb\eta \meet \theta, \eta\rb_L^\theta$ is a sublattice of $\lb\eta \meet \theta,
\eta\rb_L$.
Fix $\alpha, \beta \in \lb\eta \meet \theta, \eta\rb_L^\theta$.  We show
\begin{equation}
  \label{eq:i}
  (\alpha \join \beta) \circ \theta \subseteq \theta\circ (\alpha \join \beta),
\end{equation}
and
\begin{equation}
  \label{eq:ii}
  (\alpha \meet \beta) \circ \theta \subseteq \theta\circ (\alpha \meet \beta).
\end{equation}
The reverse inclusions follow by symmetric arguments.

Fix $(x,y) \in (\alpha \join \beta) \circ \theta$.  Then there exist $c \in X$
and $n< \omega$ such that $x \rel{(\alpha \circ^{(n)}\beta)} c \rel{\theta} y$.
Thus, $(x,y) \in \alpha \circ^{(n)}\beta \circ \theta$.  Since $\theta$ permutes
with both $\alpha$ and $\beta$, we have $(x,y) \in \theta \circ \alpha
\circ^{(n)}\beta \subseteq \theta \circ (\alpha \join \beta)$, which proves
(\ref{eq:i}).
Fix $(x,y) \in (\alpha \meet \beta) \circ \theta$.  Then 
$(x,y) \in (\alpha \circ \theta) \cap (\beta \circ \theta) = 
(\theta \circ \alpha ) \cap (\theta \circ \beta)$.  Therefore, there exist $d_1,
\,d_2$ such that 
$x \rel{\theta} d_1 \rel{\alpha} y$ and 
$x \rel{\theta} d_2 \rel{\beta} y$.  Note that $(d_1, y) \in \alpha \leq \eta$ and 
$(d_2, y) \in \beta \leq \eta$, so $(d_1, d_2) \in \eta$.  Also, $d_1
\rel{\theta} x \rel{\theta} d_2$, so $(d_1, d_2) \in \theta$.  Therefore, 
$(d_1, d_2) \in \eta \meet \theta \leq \alpha \meet \beta$.
In particular, 
$d_1 \rel{\beta} d_2 \rel{\beta} y$, so
$(d_1,y)\in \alpha \meet \beta$.
%
Thus, $x\rel{\theta} d_1 \rel{(\alpha\meet \beta)} y$, which proves~(\ref{eq:ii}).

\end{proof}


\begin{thebibliography}{1}
\providecommand{\url}[1]{{#1}}
\providecommand{\urlprefix}{URL }
\expandafter\ifx\csname urlstyle\endcsname\relax
  \providecommand{\doi}[1]{DOI~\discretionary{}{}{}#1}\else
  \providecommand{\doi}{DOI~\discretionary{}{}{}\begingroup
  \urlstyle{rm}\Url}\fi

\bibitem{Dedekind:1900}
Dedekind, R.: Ueber die von drei {M}oduln erzeugte {D}ualgruppe.
\newblock Math. Ann. \textbf{53}(3), 371--403 (1900).
\newblock \urlprefix\url{http://dx.doi.org/10.1007/BF01448979}

\bibitem{alvi:1987}
McKenzie, R.N., McNulty, G.F., Taylor, W.F.: Algebras, lattices, varieties.
  {V}ol. {I}.
\newblock Wadsworth \& Brooks/Cole, Monterey, CA (1987)

\end{thebibliography}
\def\cprime{$'$} \def\cprime{$'$}
  \def\ocirc#1{\ifmmode\setbox0=\hbox{$#1$}\dimen0=\ht0 \advance\dimen0
  by1pt\rlap{\hbox to\wd0{\hss\raise\dimen0
  \hbox{\hskip.2em$\scriptscriptstyle\circ$}\hss}}#1\else {\accent"17 #1}\fi}

\end{document}